\documentclass[11pt, oneside]{article}   	
\usepackage{geometry}                		
\usepackage{graphicx}				
\usepackage{amssymb,amsthm}					
\usepackage{natbib}
\usepackage{amsmath}
\usepackage{epstopdf}
\usepackage{bbm}
\usepackage{mathtools}

\newcommand{\dotp}[2]{\left\langle #1, #2\right\rangle}

\newcommand{\abs}[1]{\left\vert#1\right\vert}

\newcommand{\norm}[1]{\big\Vert#1\big\Vert}

\newcommand \bbP{\mathbb{P}}
\newcommand \bbE{\mathbb{E}}

\def\m{\mathcal}
\def\mb{\mathbb}
\def\mr{\mathrm}
\def\mx{\mbox}

\newtheorem{theorem}{Theorem}[section]
\newtheorem{lemma}[theorem]{Lemma}
\newtheorem{proposition}[theorem]{Proposition}
\newtheorem{corollary}[theorem]{Corollary}

\begin{document}

\title{Optimal Bayesian estimation  in random covariate design with a rescaled Gaussian process prior}

\author{ {\bf Debdeep Pati}  \\ \texttt{debdeep@stat.fsu.edu}, Department of Statistics,  Florida State University\\
        {\bf Anirban Bhattacharya} \\  \texttt{anirbanb@stat.tamu.edu}, Department of Statistics,     Texas A\& M University\\
{\bf Guang Cheng} \\ \texttt{chengg@stat.purdue.edu}, Department of Statistics, Purdue University        }

\maketitle

\begin{abstract}
In Bayesian nonparametric models, Gaussian processes provide a popular prior choice for regression function estimation. Existing literature on the theoretical investigation of the resulting posterior distribution almost exclusively assume a fixed design for covariates. The only random design result we are aware of \citep{van2011information} assumes the assigned Gaussian process to be supported on the smoothness class specified by the true function with probability one. This is a fairly restrictive assumption as it essentially rules out the Gaussian process prior with a squared exponential kernel when modeling rougher functions. In this article, we show that an appropriate rescaling of the above Gaussian process leads to a rate-optimal posterior distribution even when the covariates are independently realized from a known density on a compact set. The proofs are based on deriving sharp concentration inequalities for frequentist kernel estimators; the results might be of independent interest.

\end{abstract}


\section{Introduction}
Gaussian processes \citep{rasmussen2004gaussian,seeger2004gaussian,rasmussen2006gaussian} are widely used in the machine learning community as a principled probabilistic approach to function estimation.  A mean-zero Gaussian process is completely specified by its covariance kernel; popular choices include the squared-exponential and Mat\'{e}rn families.
Recently, there has been significant interest in frequentist convergence properties of Bayesian posteriors in Gaussian process models. \cite{ghosal2006posterior,choi2007posterior,tokdar2007posterior} established posterior consistency in a variety of settings including  nonparametric regression, classification and density estimation. \cite{seeger2008information} used an information criterion to evaluate closeness of the posterior distribution to the truth; see also \cite{van2011information}.
A major focus in the recent literature \citep{van2007bayesian,van2008rates,van2009adaptive,van2011information,aniso2014,shang2014nonparametric} has been on deriving the posterior convergence rate \citep{ghosal2000convergence}, which is defined as the minimum possible sequence $\epsilon_n \to 0$ such that for some constant $M > 0$,
\begin{align}\label{eq:post_con}
E_{\theta_0} \Pi(\norm{\theta - \theta_0} < M \epsilon_n \mid  \mathcal{D}_n) \to 1,
\end{align}
where $\mathcal{D}_n$ denotes the data, $\theta$ is the parameter of interest with some known transformation $ \Psi(\theta)$ assigned a Gaussian process prior, $\theta_0$ is the true data generating parameter and $\norm{\cdot}$ is a distance measure relevant to the statistical problem. In the context of nonparametric regression, classification and density estimation, it has been established that the posterior convergence rate based on appropriate Gaussian process priors coincides with the minimax optimal rate $n^{-\alpha/ (2 \alpha +d)}$ for $d$-variate $\alpha$-smooth functions up to a logarithmic factor \citep{van2007bayesian,van2008rates}, with rate-adaptivity to the unknown smoothness achieved in \cite{van2009adaptive,aniso2014}.

In this paper, we focus on a non-parametric regression problem,
\begin{align}\label{eq:sgp_def}
Y_i = f(X_i) + \epsilon_i, \quad \epsilon_i \sim \mx N(0, 1),
\end{align}
with $f$ assigned a mean-zero Gaussian process prior. The above-mentioned literature on posterior convergence rates under \eqref{eq:sgp_def} typically assume that the covariates $X_i$'s are fixed by design, in which the empirical $L_2$ norm $\norm{f - f_0}_{2,n} = (1/n\sum_{i=1}^n \abs{f(x_i) - f_0(x_i)}^2)^{1/2}$ is used as a discrepancy  measure in \eqref{eq:post_con}.  The empirical $L_2$ norm evaluates the discrepancy of the estimated function from the true function only at the observed data-points and is not suitable to assess out-of-sample predictive performance. In this paper, we consider a {\em random design} setup where the covariates $X_i$'s are drawn independently from a known distribution $q$, and derive the posterior convergence rates under an integrated $L_1(q)$ metric: $$\norm{f - f_0}_{1,q} = \int |f(x) - f_0(x) | q(x) dx.$$   The above integrated $L_1(q)$ metric is more relevant for studying statistical efficiency in a random design setting. From a technical standpoint, dealing with the integrated metric is challenging since one cannot directly leverage on properties of multivariate Gaussian distributions as in the case of the empirical $L_2$ norm to construct ``test functions''; a key ingredient in Bayesian asymptotics.

%

In the frequentist literature, existing results \citep{baraud2002model,brown2002asymptotic,birge2004model} on the convergence rates (with respect to an integrated metric) in random design regression require an appropriate lower bound on the smoothness of the underlying true function. For example, \cite{brown2002asymptotic,birge2004model} assumed that the univariate function $f_0$ belongs to a Lipschitz class with smoothness index $\alpha > 1/2$. Moreover, \cite{birge2004model} demonstrated the necessity of the $\alpha > 1/2$ condition by establishing a lower bound for the asymptotic risk for $\alpha \leq 1/2$. Similar lower bound condition will be assumed in our main Bayesian Theorem as well.

As far as we are aware, the only Bayesian literature considering the random design setting in (\ref{eq:sgp_def}) is \cite{van2011information} who assigned Gaussian processs with Mat\'{e}rn or squared exponential kernels. Specifically, they obtained an optimal rate $n^{-\alpha/(2 \alpha + d)}$ (up to a logarithmic factor, with respect to $L_2(q)$ norm) under a particularly strong assumption that the Gaussian process prior assigns probability one to the smoothness class containing the true function. Since the squared-exponential kernel has infinitely smooth sample paths, their result only delivers the optimal rate for analytic functions, but provides a highly suboptimal $(\log n)^{-t}$ rate for $\alpha$-smooth functions. This significantly limits the applicability of their result in the sense that it rules out the use of a squared-exponential kernel for less smooth (but more commonly used) functions.
An influential idea developed in \cite{van2007bayesian,van2009adaptive} is to scale the sample paths of a Gaussian process with a squared-exponential kernel to enable better approximation of $\alpha$-smooth functions.  The scaling is typically dependent on the smoothness of the true function and the sample size. However, Theorem 2 of  \cite{van2011information} is applicable only to priors  without scaling.  This is not evident from the statement of their theorem, but a closer inspection of their proof (ref. Page 2113) reveals that they have assumed $\tau^2:= \int \norm{f}_{\alpha | \infty}^2   d \Pi(f)$ to be a global constant for every $f$ in the support of the prior.  This may not hold for a rescaled Gaussian process.



In this article, we show that an appropriately rescaled Gaussian process prior with a squared-exponential covariance kernel leads to a rate-optimal posterior distribution (with respect to $L_1(q)$ norm) for any $\alpha$-smooth function $d$-variate $f_0$ in a random design setting  if $\alpha > d/2$. While \cite{van2011information} conjectured (see pp. 2103 after Theorem 2) that their smoothness assumption on the prior is unavoidable for the $L_2(q)$ norm, our result shows that this situation turns out to be different under the $L_1(q)$ norm. Specifically, we develop exponentially consistent test functions under the $L_1(q)$ norm using concentration inequalities for the Nadaraya--Watson kernel estimator. Existence of such test functions plays a key role in Bayesian asymptotic theory \citep{ghosal2000convergence}. For example, the classical Birg{\'e} -- Le Cam testing theory \citep{birge1984tests,le1986asymptotic} for the Hellinger metric provides appropriate tests in a wide variety of settings. \cite{gine2011rates} proposed an alternative framework for constructing tests based on concentration inequalities of frequentist estimators which is particularly useful for stronger norms; see also \cite{ray2013bayesian,PaBhPiDu2014,shang2014nonparametric}  for similar ideas in different contexts.

\section{Posterior convergence in random design regression}

\subsection{Notations}\label{sec:prelim}

Let $C[0, 1]^d$ and $C^{\alpha}[0, 1]^d$ denote the space of continuous functions and the H\"{o}lder space of $\alpha$-smooth functions $f: [0, 1]^d \to \mathbb{R}$, respectively, endowed with the supremum norm $\norm{f}_{\infty} = \sup_{ t \in [0, 1]^d } \abs{ f(t) }$. For $\alpha > 0$, the H\"{o}lder space $C^{\alpha}[0, 1]^d$ consists of functions $f \in C[0, 1]^d$ that have bounded mixed partial derivatives up to order  $\lfloor \alpha \rfloor$, with the partial derivatives of order $\lfloor \alpha \rfloor$ being Lipschitz continuous of order $\alpha - \lfloor \alpha \rfloor$. Let $\norm{\cdot}_1$ and $\norm{\cdot}_2$ respectively denote the $L_1$ and $L_2$ norm on $[0, 1]^d$ with respect to the Lebesgue measure (i.e., the uniform distribution). To distinguish the $L_2$ norm with respect to the Lebesgue measure on $\mb R^d$, we use the notation $\norm{\cdot}_{2, d}$.

We write ``$\precsim$'' for inequality up to a constant multiple. Let $\phi(t) = (2\pi)^{-1/2}\exp(-t^2/2)$ denote the standard normal density, and let $\phi_n(x) = \prod_{i=1}^n \phi(x_i)$ for $x \in \mb R^n$. Let a star denote a convolution, i.e., $f_1 \star f_2(y) = \int f_1(y-t) f_2(t) dt$. 
We denote the Fourier transform of $f$, whenever defined, by $\hat{f}$, with $\hat{f}(\lambda) = (2\pi)^{-d}\int \exp( i \dotp{\lambda}{t} ) f(t) dt$, where $\dotp{\lambda}{t}$ denotes the complex inner product. Under this convention, the inverse Fourier transform $f(t) = \int \exp( -i \dotp{\lambda}{t} ) \hat{f}(\lambda) d\lambda$ and $\hat{h} = (2 \pi)^d \hat{f} \hat{g}$ when $h = f \star g$.

Throughout $C, C', C_1, C_2, \ldots$ are generically used to denote positive constants whose values might change from one line to another, but are independent from everything else.  $Z_{1:n}$ is used as a shorthand for $Z_1, \ldots, Z_n$.

In the sequel, we consider a Gaussian process prior $\Pi$ on the regression function $f$ with $\bbE f(x) = 0$ and covariance kernel $c(x, x') = \mbox{cov}(f(x), f(x'))$. In particular, we focus on the squared-exponential kernel $c_a(x, x') = \exp(- a^2 \norm{x - x'}^2)$ indexed by an ``inverse-bandwidth'' parameter $a$. We next recall some important facts relevant to our setting from \cite{van2009adaptive} regarding the spectral measure and reproducing kernel Hilbert space of Gaussian process priors.
For the squared-exponential kernel $c_a$, the spectral measure $\mu_a$ admits a density $\omega_a$ with respect to Lebesgue measure, where $\omega_a(\lambda) = a^{-d} \omega(\lambda/a)$, with $\omega(\lambda) = \exp(- \| \lambda \|^2/4)/(2^d \pi^{d/2})$. The reproducing kernel Hilbert space $\mb H^a$  associated with a Gaussian process prior $\Pi$ consists of (real parts of) functions $h(t) = \int \exp( i \dotp{\lambda}{t} ) \xi(\lambda) d \mu_a(\lambda)$, where $\mu_a$ is the spectral measure of $\Pi$ and $\xi \in L_2(\mu_a)$. The squared Hilbert space norm of $h$ above is given by $\norm{h}_{\mb H^a}^2 = \norm{\xi \omega_a^{1/2}}_{2, d}^2 = \int \xi^2(\lambda) \omega_a(\lambda) d\lambda$; let $\mb H_1^a$ denote the unit ball of the reproducing kernel Hilbert space $\{h \in \mb H^a: \norm{h}_{\mb H^a} \leq 1\}$. Finally, let $\mb B_1$ denote the unit ball of $C[0, 1]^d$ with respect to the supremum norm.
For a detailed review of reproducing kernel Hilbert space of Gaussian process priors and connections with posterior contraction rates, kindly refer to \cite{van2008reproducing}.

\subsection{Main result}\label{sec:main}

Consider the nonparametric regression model \eqref{eq:sgp_def}.
We assume a random design setup, where given the regression function $f: [0, 1]^d \to \mb R$, the data $(X_1, Y_1), \ldots, (X_n, Y_n)$ are independently generated, with $X_i$ having a density $q$ on $[0, 1]^d$ that is bounded away from zero and infinity. Let $q(y, x) = q(y \mid x) q(x)$ denote the joint density of $(Y, X)$ given $f$, where $q(y \mid x) = \phi\{y - f(x)\}$. The joint data likelihood given $f$ is therefore
$$
q^{(n)}(Y_{1:n}, X_{1:n} \mid f) = \prod_{i=1}^n q(Y_i, X_i) = \prod_{i=1}^n \phi\{Y_i  - f(X_i)\}  q(X_i).
$$
Similarly, we define $q^{(n)}(Y_{1:n} \mid X_{1:n}, f)$ and $q^{(n)}(X_{1:n})$ as the density of $(Y_{1:n} \mid X_{1:n}, f)$ and $X_{1:n}$ respectively. Let $\bbE^f_{Y, X} (\bbP^f_{Y, X})$ denote an expectation (probability) with respect to $q^{(n)}(Y_{1:n}, X_{1:n} \mid f)$. Similarly define $\bbE_{Y \mid X}^f (\bbP_{Y \mid X}^f)$ and $\bbE_X^f (\bbP_X^f)$. When $f$ is clear from the context, we shall drop it from the superscript.

We assume a mean zero Gaussian process prior $\Pi$ on $f$ with a squared exponential kernel $\exp(- a_n^2 \norm{x - x'}^2)$ and denote the corresponding posterior measure by $\Pi(\cdot \mid Y_{1:n}, X_{1:n})$, so that
$$
\Pi(f \mid Y_{1:n}, X_{1:n}) \propto q^{(n)}(Y_{1:n} \mid X_{1:n}, f) \, \Pi(f).
$$
Assuming the true regression function is $f_0$, we study concentration of the posterior $\Pi(\cdot \mid Y_{1:n}, X_{1:n})$ in an $L_1(q)$ neighborhood of $f_0$. \begin{theorem}\label{thm:main}
 Assume that $f_0 \in C^{\alpha}[0, 1]^d$ with $\alpha > d/2$ and $\Pi$ is a mean-zero Gaussian process prior with a squared exponential covariance kernel $c(x, x') = \exp(- a_n^2 \norm{x - x'}^2)$. Set $a_n = n^{1/(2 \alpha + d)}$.  Then with $\epsilon_n =n^{-\alpha/ (2 \alpha +d)} \log^{3t_1/2}n$ for  $t_1 \geq (d+1)/2$, and some fixed sufficiently large constant $M > 0$,
\begin{eqnarray}
\bbE_{Y, X}^{f_0} \Pi \big(\norm{f - f_0}_{1,q} > M \epsilon_n \mid Y_{1:n}, X_{1:n} \big) \to 0.
\end{eqnarray}
\end{theorem}
As stated previously, the condition $\alpha > d/2$ is necessary to obtain the optimal rate.  \cite{van2009adaptive} showed that the squared-exponential covariance kernel without rescaling leads to a very slow $(\log n)^{-l}$ contraction rate for $\alpha$-smooth functions both in the fixed and random design settings. This is not surprising as the sample paths of such a GP are analytic. The effect of scaling the prior using the ``inverse bandwidth'' $a$ to yield the optimal posterior concentration was first noted by \cite{van2007bayesian} in a fixed design context, who showed (for $d = 1$) that a deterministic scaling $a_n = n^{1/(2\alpha+1)}$ produces priors that are suitable for modeling $\alpha$-smooth functions. Theorem \ref{thm:main} assures that the same rescaling idea continues to work in the {\it random} design setting for an integrated $L_1$ norm.

The optimal rescaling in Theorem \ref{thm:main} requires knowledge of the true smoothness $\alpha$. If there is a mismatch between the prior regularity and the function smoothness, one would typically expect a sub-optimal rate. Corollary \ref{cor:main} quantifies this fact; while we only derive an upper bound to the posterior convergence rate, such bounds are usually tight \citep{van2009adaptive}. In absence of any prior knowledge regarding the smoothness, one may resort to an empirical or fully Bayes approach as in \cite{van2009adaptive,szabo2013empirical}. The related theoretical investigation will be considerably harder in such cases.

\begin{corollary}\label{cor:main}
Under the conditions of Theorem \ref{thm:main}, if $a_n = n^{1/ (2\beta +d)}$ for $\beta > d/2, \beta \neq \alpha$, the conclusion of Theorem \ref{thm:main} holds with $\epsilon_n =  n^{-\alpha \wedge \beta/ (2\beta+d)} \log^{3t_1/2}n$ for  $t_1 \geq d/(4 - 2\kappa)$ for any $0 < \kappa < 2$. 
\end{corollary}
Observe that the optimal rate $n^{-\alpha/(2\alpha+d)}$ is attained (upto logarithmic terms) if and only if $\alpha = \beta$. A scaling $n^{1/ (2\beta + d)}$ for $\beta < \alpha$ makes the prior rougher compared to the true function while $\beta > \alpha$ renders smoother prior realizations. In both cases, sub-optimal rates are obtained.   This is in accordance with the findings for GP priors with Mat\'{e}rn covariance kernel; refer to Theorem 5 in \cite{van2011information}.

Note that by taking $\kappa$ very close to $0$, we can improve on the power of the $\log n$ term in Corollary \ref{cor:main}  from that in Theorem \ref{thm:main}.
The  difference in the power of  $\log n$  stems from the fact that the corollary only targets sub-optimal rates as opposed to Theorem \ref{thm:main}. Hence a portion of the power of $\log n$ can be eliminated in Corollary \ref{cor:main}.

\subsection{Contributions beyond literature}
The proof of Theorem \ref{thm:main} follows from a general set of sufficient conditions for posterior concentration in model \eqref{eq:sgp_def}; kindly refer to Theorem \ref{thm:main_cond} stated in the next Section. In particular, we exploit
concentration inequalities for suitable kernel estimators to construct the aforementioned exponentially consistent sequence of test functions.  Such techniques have been used previously to show convergence rates in density estimation \citep{gine2011rates} and in linear inverse problems \citep{ray2013bayesian}. Their techniques do not directly apply to our case partly due to the lack of concentration bounds for kernel based estimators. \cite{gine2011rates,ray2013bayesian} construct estimators based on truncated spectral representations which are well suited to sieve priors.  However, to deal with a Gaussian process prior with a squared-exponential covariance kernel, we need to construct test functions based on the Nadaraya--Watson kernel estimator and derive sharp concentration bounds for this class of estimators in Lemma \ref{lem:devygx} and \ref{lem:devx}.


The choice of the norm dictating the neighborhood around the true parameter plays a critical role in Bayesian asymptotics. A fundamental tool for relating the likelihood ratio with the neighborhood in consideration is a sequence of exponentially consistent test functions \citep{ghosal2000convergence}. In a regression setting, such test functions are guaranteed to exist for the empirical $L_2$  norm by exploiting a direct relation between the  empirical $L_2$  norm and the likelihood ratio of  the multivariate Gaussian densities involved; refer to Section 4 of \cite{van2011information}. However, the integrated norm involves covariate points different from the observations, which makes the problem more challenging. \cite{van2011information} applied Bernstein's inequality to extrapolate to the $L_2(q)$ norm from the empirical $L_2$ norm. However, as stated in the Introduction, their approach only works for priors that are supported on the true smoothness class with probability one.



Among other related work, Section 4 of \cite{kleijn2006misspecified} considers random design regression, where a correspondence between the Kullback--Leibler and $L_2(q)$ neighborhood is established to derive the test function, assuming the prior support consists of uniformly bounded functions. However, this assumption does not hold for the rescaled Gaussian process prior in Theorem \ref{thm:main}. In particular, the {\em sieves} constructed in \cite{van2007bayesian} of the form $M_n \mb H_1^{a_n} + \epsilon_n \mb B_1$ with $M_n \to \infty$ do not correspond to sup-norm bounded subsets of $C[0, 1]^d$.

We comment here that convergence in the integrated metric has been settled in the binary regression setting. Using a logistic link function, a direct agreement can be established between the integrated $L_1$ metric on the function space and the Hellinger distance between the resulting densities arising from the Bernoulli likelihood; see for example, Section 3.2 of \cite{van2008rates}. Second, in this paper we implicitly refer to Gaussian processs which are specified by a kernel function; specifically, kernel functions which do not admit a finite series representation, such as the squared-exponential kernel. If a Gaussian process is specified via a truncated orthogonal series representation with independent Gaussian priors on the coefficients, the integrated metric can be related to the $L_2$ norm of the coefficient vector \citep{bontemps11}.

\section{Auxiliary results} \label{sec:aux}
We now state a general theorem which presents a set of sufficient conditions for proving Theorem \ref{thm:main}. From now onwards, we shall assume the covariate distribution $q$ to be a uniform distribution on $[0, 1]^d$ for notational simplicity; modifying our construction to a general $q$, which is bounded from above and below, is straightforward. The $L_1(q)$ norm $\norm{\cdot}_{1,q}$ with $q$ the uniform distribution on $[0, 1]^d$ shall be simply denoted by $\norm{\cdot}_1$ following our convention in Section 2.1.
A proof of Theorem \ref{thm:main_cond} can be found in the Appendix.
\begin{theorem}\label{thm:main_cond}
Let $\epsilon_n', \delta_n$ be sequences such that $\epsilon_n', \delta_n \to 0$ and $n\epsilon_n'^2, n \delta_n^2 \to \infty$. Let $U_n = \{ f: \| f - f_0 \|_1 > M \epsilon_n'\}$ for some fixed $M > 0$. Suppose that there exists a sequence of estimators $\tilde{f}_n$ for $f$ based on $(Y_{1:n}, X_{1:n})$ and a sequence of subsets/sieves $\m P_n$ of $C[0, 1]^d$ such that
\begin{align}
& \Pi(\m P_n^c) \leq \exp\{-(C+4) n \delta_n^2\}, \tag{PCS} \\
& \norm{\bbE_{Y,X}^{f_0} \tilde{f}_n - f_0}_1 < \epsilon_n', \tag{BT} \\
& \bbP_{Y, X}^{f_0} \bigg( \norm{\tilde{f}_n - \bbE_{Y, X}^{f_0} \tilde{f}_n}_1 > \epsilon_n' \bigg) \leq \exp\{-(C + 4) n \delta_n^2\}, \tag{DT} \\
& \sup_{f \in \m P_n \cap U_n} \norm{\bbE_{Y,X}^{f} \tilde{f}_n - f}_1 < \epsilon_n', \tag{BS} \\
& \sup_{f \in \m P_n \cap U_n} \bbP_{Y, X}^{f} \bigg( \norm{\tilde{f}_n - \bbE_{Y, X}^{f} \tilde{f}_n}_1 > \epsilon_n' \bigg) \leq \exp\{-(C + 4) n \delta_n^2\}, \tag{DS} \\
& \Pi \bigg(\norm{f - f_0}_{\infty} \leq \delta_n \bigg) \geq \exp\{-n \delta_n^2\}. \tag{PCN}
\end{align}
Then, $\bbE_{Y, X}^{f_0} \Pi \big(U_n \mid Y_{1:n}, X_{1:n} \big) \to 0$.
\end{theorem}
Condition (PCS) implies that the prior probability of the complement of the sieve $\m P_n$ is exponentially small. Condition (BT) assumes a sufficiently accurate estimator $\tilde{f}_n$ with bias smaller than $\epsilon_n$ at $f_0$ while (DT) assumes an exponential concentration bound of $\tilde{f}_n$ from its expectation under $q^{(n)}(\cdot \mid f_0)$. (BS) and (DS) assume similar conditions as (BT) and (DT) under $q^{(n)}(\cdot \mid f)$ for any $f \in \m P_n \cap U_n$. The conditions (BT), (DT); (BS), (DS) jointly guarantee the existence of exponentially consistent test functions; see Lemma \ref{lem:test} in the Appendix. Condition (PCN) assumes that the prior $\Pi$ places ``enough'' mass in an $\epsilon_n$-neighborhood of the truth $f_0$ in terms of the sup-norm.

\subsection{Verifying the conditions of Theorem \ref{thm:main_cond} to prove Theorem \ref{thm:main}}

Letting $\delta_n = \epsilon_n' = \epsilon_n$ with $\epsilon_n'$ and $\epsilon_n$ as in the statement of Theorem \ref{thm:main_cond} and Theorem \ref{thm:main} respectively, we now proceed to construct $\m P_n$ and $\tilde{f}_n$ that satisfy the conditions of Theorem \ref{thm:main_cond}. While we choose the same sieve as in \cite{van2007bayesian}, part of the technical challenge lies in the fact that the concentration bounds need to be derived not just for the truth, but rather for every function in the sieve. This requires precise control on the {\em size} of the functions in the sieve $\m P_n$. We show in Proposition \ref{prop:l2normbd} below that the functions in the chosen sieve are uniformly bounded in $L_2$ norm, although they are unbounded in the supremum norm.

Let $\psi : \mb R^d \to \mb C$ be a function such that $\int \psi(t) dt = 1$, $\int t^k \psi(t) dt = 0$ for any non-zero multi-index $k = (k_1, \ldots, k_d)$ of non-zero integers, $\int |t|^{\max\{\alpha,2\}} |\psi(t)| dt < \infty$, and the functions $|\hat{\psi}|/\omega$ and $|\hat{\psi}|^2/\omega$ are uniformly bounded; see proof of Lemma 4.3 in \cite{van2009adaptive}. Define,
\begin{align}\label{eq:fnhat}
\tilde{f}_n(x) = \frac{1}{n} \sum_{i=1}^n \psi_{\sigma_n}(x - X_i) Y_i,
\end{align}
where $\psi_{\sigma}(t) = \sigma^{-d} \psi(t/\sigma)$ for $\sigma > 0$ and set $\sigma_n = n^{-1/(2\alpha + d)} \log^{-t_2}n$,  $t_2 = 1/(2- \kappa)$ for some $0< \kappa < 1$ . Next, with $M_n = a_n^{d/2}$, set
\begin{align}\label{eq:sieve}
\m P_n = M_n \mb H_1^{a_n} + \epsilon_n \mb B_1.
\end{align}

\noindent Assume $f_0 \in C^{\alpha}[0,1]^d$. Let $\tilde{f}_n$ and $\m P_n$ be as in \eqref{eq:fnhat} and \eqref{eq:sieve} respectively. We show below that the conditions of Theorem \ref{thm:main_cond} are satisfied with $\epsilon_n = n^{-\alpha/(2 \alpha + d)} \log ^{t_1} n$ for $t_1 \geq \max \{ t_2d/2, (d+1)/ 2\}=(d+1)/2$, provided $\alpha > d/2$.
(PCS) follow from the proof of Theorem 3.1 in \cite{van2009adaptive}. To verify (PCN), observe from the proof of Theorem 3.1 in \cite{van2009adaptive} that with $a_n = n^{1/ (2\alpha + d)}$,
\begin{eqnarray*}
\Pi(\norm{f - f_0}_{\infty} \leq \delta_n) \geq \exp\{- n^{d/(2\alpha + d)} (\log n)^{d+1}\}
\end{eqnarray*}
for $\delta_n \geq n^{-\alpha/ (2\alpha + d)}$.  Hence  (PCN) is satisfied with $\delta_n = \epsilon_n$.

We verify (BS) and (DS); the verifications of (BT) and (DT) follow along similar lines.

{\em We first show that (DS) holds.} Fix $f \in \m P_n \cap U_n$. We drop the superscript $f$ from $\bbE^f$ in the sequel. Let $f_n(x) = \psi_{\sigma_n} \star f(x) = \int \psi_{\sigma_n}(x-t) f(t) dt$ and $f_n^X(x) = n^{-1} \sum_{i=1}^n \psi_{\sigma_n}(x - X_i) f(X_i)$. Observe that $\bbE_{Y, X} \tilde{f}_n = f_n$ and $\bbE_{Y \mid X} \tilde{f}_n = f_n^X$. Then,
\begin{align}\label{eq:decomp}
& \bbP_{Y, X}\bigg( \norm{\tilde{f}_n - \bbE_{Y, X} \tilde{f}_n}_1 > \epsilon_n \bigg)
= \bbP_{Y, X} \bigg(\norm{\tilde{f}_n - f_n}_1 > \epsilon_n \bigg) \nonumber \\
& \leq \bbP_{Y, X} \bigg(\norm{\tilde{f}_n - f_n}_1 > \epsilon_n, \norm{f_n^X -f_n}_1 \leq \epsilon_n/2 \bigg) + \bbP_{X} \bigg(\norm{f_n^X - f_n}_1 \geq \epsilon_n/2 \bigg) \nonumber \\
& \leq \bbE_X \bbP_{Y \mid X} \bigg(\norm{\tilde{f}_n - f_n^X}_1 \geq \epsilon_n/2 \mid X_{1:n} \bigg) + \bbP_{X} \bigg(\norm{f_n^X - f_n}_1 \geq \epsilon_n/2 \bigg).
\end{align}
Lemmata \ref{lem:devygx} and \ref{lem:devx} below deliver the desired bounds for the two terms appearing in \eqref{eq:decomp}.
\begin{lemma}\label{lem:devygx}
Under conditions of Theorem \ref{thm:main},
\begin{align*}
\bbP_{Y \mid X} \bigg(\norm{\tilde{f}_n - f_n^X}_1 \geq \epsilon_n/2 \mid X_{1:n} \bigg) \leq \exp(-Cn \epsilon_n^2) \quad \mbox{a.s.}
\end{align*}
for some constant $C > 0$.
\end{lemma}

\begin{proof}
For simplicity of notation, we suppress the term  ``a.s." in the displays that follow.  Let $T(x) = n (\tilde{f}_n - f_n^X)(x) = \sum_{i = 1}^n \psi_{\sigma_n}(x-X_i) Z_i$, where $Z_i = Y_i - f(X_i)$ with $Z_{1:n} \mid X_{1:n}, f \sim \mx N_n(0, \mr I)$. Given $X_{1:n}$, $T$ is a random element of $L_1[0, 1]^d$ and $\|T\|_1$ is a non-negative random variable. By the Hahn--Banach theorem, there exists a bounded linear functional $G$ on $L_{\infty}[0, 1]^d$ such that $G(h) = \int T(x) h(x) dx$ for all $h \in L_{\infty}[0, 1]^d$ and $\|T\|_1 = \|G\|_{\m F}$, where $\|G\|_{\m F} = \sup_{h \in \m F} | G(h)|$ and $\m F$ is a countable dense subset of $\{h \in L_{\infty}[0, 1]^d : \| h\|_{\infty} \leq 1\}$.

By definition, $G(h) = \sum_{i=1}^n a_i Z_i$, where $a_i = \int \psi_{\sigma_n}(x - X_i) h(x) dx$. Thus, given $X_{1:n}$, $\{ G(h) : h \in \m F\}$ is a Gaussian process and
\begin{align}\label{eq:hbi}
\bbP_{Y \mid X} \bigg(\norm{\tilde{f}_n - f_n^X}_1 \geq \epsilon_n/2 \mid X_{1:n} \bigg) = \bbP_{Y \mid X} \bigg(\norm{G}_{\m F} \geq n\epsilon_n/2 \mid X_{1:n} \bigg).
\end{align}
By Borell's inequality \citep{adler1990introduction},
\begin{align}\label{eq:borell}
\bbP_{Y \mid X} \big(\norm{G}_{\m F} - \bbE_{Y \mid X} \norm{G}_{\m F}  \geq t \mid X_{1:n} \big) \leq 2 \exp\{-t^2/(2 \sigma_{\m F}^2)\},
\end{align}
where $\sigma_{\m F}^2 = \sup_{h \in \m F} \bbE_{Y \mid X} G(h)^2$. We now proceed to estimate $\sigma_{\m F}^2$ and $\bbE_{Y \mid X} \norm{G}_{\m F}$. For any $h \in \m F$,
\begin{align}
& \bbE_{Y \mid X} G(h)^2 = \sum_{i=1}^n \bigg\{\int_{[0, 1]^d} \psi_{\sigma_n}(x - X_i) h(x) dx \bigg\}^2  \nonumber \\
& \leq \norm{h}_{\infty}^2 \sum_{i=1}^n \bigg\{\int_{[0, 1]^d} \abs{\psi_{\sigma_n}(x - X_i)}dx \bigg\}^2 \leq C_1 n, \label{eq:var}
\end{align}
where $C_1 = \int |\psi(t)| dt$. Hence, $\sigma_{\m F}^2 \leq C_1 n$.

We next bound $\bbE_{Y \mid X} \norm{G}_{\m F} = \bbE_{Y \mid X} \norm{T}_1$. By Jensen's inequality, $\bbE_{Y \mid X} \norm{T}_1 \leq (\bbE_{Y \mid X} \norm{T}_1^2)^{1/2}$. Further,
\begin{align*}
(\bbE_{Y \mid X} \norm{T}_1^2)^{1/2} = \bigg[ \int \bigg\{\int |T(x)| dx \bigg\}^2 \phi_n(z) dz \bigg]^{1/2}  \leq \int \bigg\{ \int T(x)^2 \phi_n(z) dz\bigg\}^{1/2} dx,
\end{align*}
where the above inequality follows from an integral version of Minkowski's inequality. Recalling $T(x) = \sum_{i = 1}^n \psi_{\sigma_n}(x-X_i) Z_i$, $\int T(x)^2 \phi_n(z) dz = \bbE [ T(x)^2 \mid X_{1:n}] = \sum_{i=1}^n \psi_{\sigma_n}(x - X_i)^2$. Substituting this in the above display and using Jensen's inequality one more time, we get
\begin{align}
& \bbE_{Y \mid X} \norm{T}_1 \leq \int \bigg\{\sum_{i=1}^n \psi_{\sigma_n}(x - X_i)^2 \bigg\}^{1/2} dx\nonumber \\
& \leq \bigg\{ \sum_{i=1}^n \int \psi_{\sigma_n}(x - X_i)^2 dx \bigg\}^{1/2} \leq C_2 (n/\sigma_n^d)^{1/2},
\end{align}
where $C_2 = \{ \int \psi(t)^2 dt \}^{1/2}$.
In \eqref{eq:borell}, set $t = n \epsilon_n/4$. From the above calculations, $\bbE_{Y \mid X} \norm{G}_{\m F} \leq C_2 (n/\sigma_n^d)^{1/2} \leq n \epsilon_n/4$ since $t_1 \geq t_2d/2$. Using $\sigma_{\m F}^2 \leq C_1 n$, we finally obtain $\bbP_{Y \mid X} \bigg(\norm{G}_{\m F} \geq n\epsilon_n/2 \mid X_{1:n} \bigg) \leq 2 \exp(-C n \epsilon_n^2)$.
\end{proof}

\begin{lemma}\label{lem:devx}
Under conditions of Theorem \ref{thm:main},
\begin{align*}
\bbP_{X} \bigg(\norm{f_n^X -f_n}_1 \geq \epsilon_n/2 \bigg) \leq \exp(-n \epsilon_n^2).
\end{align*}
\end{lemma}

\begin{proof}
As in Lemma \ref{lem:devygx}, we express the desired probability in terms of a tail bound for the supremum of a stochastic process. However, the stochastic process in this case is no longer a Gaussian process and we cannot use Borell's inequality here. We instead use Bosquet's version of Talagrand's inequality for the supremum of a centered empirical process. The following Proposition \ref{prop:talagrand} is adapted from \cite{bousquet2003concentration} which also appears in Section 3.1 of \cite{gine2011rates}.
\begin{proposition}\label{prop:talagrand}
Assume $X_1, \ldots, X_n$ are independent and identically distributed as $P$. Let $\m G$ be a countable set of real valued functions and assume all functions $g \in \m G$ are $P$-measurable, square integrable and satisfy $\bbE_P[g] = 0$. Assume $K_1 = \sup_{g \in \m G} \norm{g}_{\infty} < \infty$ and let $W = \sup_{g \in \m G} \abs{\sum_{i=1}^n g(X_i)}$. Further, let $\sigma_{\m G}^2 = \sup_{g \in \m G} \bbE_P[g(X_1)^2]$ and $K_2 = n \sigma_{\m G}^2 + K_1\bbE_P[W]$. Then, for any $t > 0$,
\begin{align*}
\bbP\bigg\{W \geq \bbE_P W + (2 K_2 t)^{1/2} + \frac{K_1t}{3}\bigg\} \leq \exp(-t).
\end{align*}
\end{proposition}
Let $L_x(t) = \psi_{\sigma_n}(x-t) f(t) - \psi_{\sigma_n} \star f(x)$ for $x, t \in [0, 1]^d$ and $W = \int_{[0,1]^d} | \sum_{i=1}^n L_x(X_i) | dx$. Clearly, $\bbP_X(\norm{f_n^X - f_n}_1 > \epsilon_n/2) = \bbP_X(W > n \epsilon_n/2)$. By an application of Hahn--Banach theorem as in the proof of Lemma \ref{lem:devygx}, $W = \norm{G}_{\m F}$, where $\m F$ is a countable dense subset of the unit ball of $L_{\infty}[0, 1]^d$, $G(h) = \sum_{i=1}^n g(X_i)$, and $g(t) = \int_{[0, 1]^d} L_x(t) h(x) dx$. Letting $\m G$ denote the class of functions $\{ g(t) = \int_{[0, 1]^d} L_x(t) h(x) dx, \, h \in \m F\}$, one has $\norm{G}_{\m F} = \sup_{g \in \m G} |\sum_{i=1}^n g(X_i)|$. Putting together, $W = \sup_{g \in \m G} |\sum_{i=1}^n g(X_i)|$ and $\bbE_X g(X_1) = 0$ by Tonelli's theorem. We now aim to apply Proposition \ref{prop:talagrand} to bound $\bbP_X(W > n\epsilon_n/2)$. In order to apply Proposition  \ref{prop:talagrand}, we need to estimate $K_1, \sigma_{\m G}^2, K_2$ and $\bbE_P(W)$ which is carried out below.

Fix $g \in \m G$. Then, there exists $h \in \m F$ such that $g(t) = \int_{[0, 1]^d} L_x(t) h(x) dx = f(t) \int_{[0, 1]^d} \psi_{\sigma_n}(t-x) h(x) dx - \int \psi_{\sigma_n} \star f(x) \, h(x) dx$. Using the triangle inequality,
$$|g(t)| \leq |f(t)| \int_{[0, 1]^d} | \psi_{\sigma_n}(t-x) h(x)| dx + \int_{[0, 1]^d} |\psi_{\sigma_n} \star f(x)| \, | h(x) | dx.$$
Using $\|h\|_{\infty} \leq 1$, the first term in the above display can be bounded above by $C_1 \|f\|_{\infty}$ where $C_1 = \int |\psi(t)| dt$. Similarly, the second term can be bounded above by $\| \psi_{\sigma_n} \star f\|_1 \leq \| \psi_{\sigma_n} \star f\|_{\infty} \leq \|f\|_{\infty} + \epsilon_n$, where the final inequality follows from (BS). Noting that for any $f \in \m P_n$, $\|f\|_{\infty} \leq 2 M_n$ (since the Hilbert space norm is stronger than the $\|\cdot\|_{\infty}$ norm), we have $K_1 \leq C M_n$.

Next we bound $\sigma_{\m G}^2=  \sup_{g \in \m G} \int_{[0,1]^d} g(t)^2 dt$. Fix $g \in \m G$. Using the expression for $g(t)$ in the previous paragraph, $|g(t)| \leq |f(t)| \int |\psi_{\sigma_n}(x-t)| dx + \int |\psi_{\sigma_n} \star f(x) | dx$. As before, we can bound $\int |\psi_{\sigma_n}(x-t)| dx$ from above by $C_1$ and also $\int |\psi_{\sigma_n} \star f(x) | dx \leq C_1 \int_{s \in [0, 1]^d} |f(s)| ds$. Using $(|a| + |b|)^2 \leq 2(|a|^2 + |b|^2)$ and the Cauchy--Schwarz inequality, $|g(t)|^2 \leq C |f(t)|^2 + C \{\int_{s \in [0, 1]^d} |f(s)| ds\}^2 \leq C \big\{ |f(t)|^2 + \|f\|_2^2\big\}$. Thus, we have $\sigma_{\m G}^2 \leq C  \|f\|_2^2$ for some absolute constant $C$. Using the bound for $\sup_{f \in \m P_n} \|f\|_2^2$ in the following Proposition \ref{prop:l2normbd}, we conclude that $\sigma_{\m G}^2 \leq C$ for some absolute constant $C>0$.
\begin{proposition}\label{prop:l2normbd}
Recall $\m P_n$ from \eqref{eq:sieve}. Then, $\sup_{f \in \m P_n} \|f\|_2^2 \leq C$ for some absolute constant $C> 0$.
\end{proposition}
\begin{proof}
Let $f \in \m P_n$. Then, there exists $h \in \mb H^{a_n}$ with $\|h\|_{\mb H^{a_n}} \leq M_n$ such that $\|f - h\|_{\infty} \leq \epsilon_n$. Hence, $\|f\|_2^2 \leq 2(\|h\|_2^2 + \epsilon_n^2)$ and it is enough to bound $\|h\|_2^2$.
Recalling that $\| \cdot \|_{2,d}$ denotes the $L_2$ norm of $\mb R^d$, we have $\|h\|_2^2 \leq \|h\|_{2,d}^2$. We provide a bound for $\|h\|_{2,d}^2$ below.

There exists $\psi \in L_2(\mu_{a_n})$ such that $h(t) =  \int \exp( i \dotp{\lambda}{t} ) \xi(\lambda) \omega_{a_n}(\lambda) d \lambda$. Letting $\hat{h}$ denote the Fourier transform of $h$, one has from the Fourier inversion theorem that $\hat{h}(\lambda) = \xi(-\lambda) \omega_{a_n}(\lambda)$. By Parseval's theorem, $\|h\|_{2,d}^2 = \|\hat{h}\|_{2,d}^2 = \int \xi^2(\lambda) \omega_{a_n}^2(\lambda) d \lambda$\footnote{$\omega_{a_n}$ is symmetric about zero}. Observe that $\omega_{a_n}^2(\lambda) = a_n^{-2d} \exp\{-\|\lambda\|^2/(2a_n^2)\}/C^2$, where $C = 2^d \pi^{d/2}$. Hence,
\begin{align*}
& \|h\|_{2,d}^2 = \frac{a_n^{-2d}}{C^2} \int \xi^2(\lambda) \exp\{-\|\lambda\|^2/(2a_n^2)\} d\lambda \leq \frac{a_n^{-2d}}{C^2} \int \xi^2(\lambda) \exp\{-\|\lambda\|^2/(4 a_n^2)\} d\lambda\nonumber \\
& = \frac{a_n^{-d}}{C} \int \xi^2(\lambda) \omega_{a_n}(\lambda) d\lambda = \frac{\|h\|_{\mb H^{a_n}}^2}{C a_n^d} \leq \frac{M_n^2}{C a_n^d} = \frac{1}{C},
\end{align*}
since $\|h\|_{\mb H^{a_n}}^2 = \norm{\xi \omega_{a_n}^{1/2}}_{2,d}^2$ and $M_n = a_n^{d/2}$.
\end{proof}
Finally, we proceed to bound $\bbE_X W$, where $W = \int_{[0,1]^d} | \sum_{i=1}^n L_x(X_i) | dx$. Using Jensen's inequality and the integral version of Minkowski's inequality, one has
\begin{align}
& \bbE_X W \leq (\bbE_X W^2)^{1/2} = \bigg[\int_{\prod_{i=1}^n [0,1]^d} \bigg\{ \int_{[0,1]^d} | \sum_{i=1}^n L_x(t_i) dx | \bigg\}^{2} dt_1\ldots dt_n\bigg]^{1/2} \nonumber \\
& \leq \int_{[0,1]^d} \bigg\{ \int_{\prod_{i=1}^n [0,1]^d} | \sum_{i=1}^n L_x(t_i) |^2 dt_1 \ldots dt_n \bigg\}^{1/2} dx \nonumber.
\end{align}
Clearly, $\int_{\prod_{i=1}^n [0,1]^d} | \sum_{i=1}^n L_x(t_i) |^2 dt_1 \ldots dt_n = \mbox{Var}_X\{\sum_{i=1}^n L_x(X_i)\} = n \mbox{Var}_X\{L_x(X_1)\}$, since $\bbE_X L_x(X_1) = 0$. Further, $\mbox{Var}_X\{L_x(X_1)\} \leq \bbE_X \big\{ \psi_{\sigma_n}(x-X_1) f(X_1)\big\}^2 = \int_{[0, 1]^d} \psi_{\sigma_n}(x-t)^2 f(t)^2 dt \leq  \frac{1}{\sigma_n^d} \psi_{\sigma_n} \star f^2$. Substituting this in the above display
\begin{align*}
\bbE_X W  & \leq \bigg(\frac{n}{\sigma_n^d}\bigg)^{1/2} \int_{[0,1]^d} \big\{\psi_{\sigma_n} \star f^2(x)\big\}^{1/2} dx \\
 &\leq \bigg(\frac{n}{\sigma_n^d}\bigg)^{1/2} \bigg\{\int_{[0,1]^d}\abs{ \psi_{\sigma_n}} \star f^2(x) dx \bigg\}^{1/2} \\
&\leq \bigg(\frac{n}{\sigma_n^d}\bigg)^{1/2} \bigg\{\int_{[0,1]^d} \int_{[0,1]^d} \abs{\psi_{\sigma_n} (x- t)} f^2(t) dt dx \bigg\}^{1/2} \\
& \leq \bigg(\frac{n}{\sigma_n^d}\bigg)^{1/2} \bigg[\int_{[0,1]^d} f^2(t) \bigg\{\int_{\mb{R}^d} \abs{\psi_{\sigma_n} (x- t)} dx\bigg\} dt \bigg]^{1/2} \\
&\leq C \bigg(\frac{n}{\sigma_n^d}\bigg)^{1/2} = C n^{\frac{\alpha + d}{2\alpha + d}} \log^{t_2 d/2} n \leq C n\epsilon_n.
\end{align*}
From the penultimate line to the last line of the above display, we invoked Proposition \ref{prop:l2normbd} to bound $\norm{f}_2$ by a constant.
We have thus obtained  $K_1 \leq C M_n$ and $K_2 \leq C n$. In Proposition \ref{prop:talagrand}, set $t = n \epsilon_n^2$. We have $K_1 t \leq C (n \epsilon_n M_n) \epsilon_n \leq n \epsilon_n$ for sufficiently large $n$ provided $\alpha > d/2$. Further,
$K_2 t \leq  n^2 \epsilon_n^2 + K_1\bbE_P(W) t =  n^{\frac{2\alpha + 2d}{2\alpha + d}} \log ^{3t_1}n+  n^{\frac{\alpha + 2d + d/2}{2\alpha + d}} \log^{2t_1 + t_2d/2} n \leq 2n^{\frac{2\alpha + 2d}{2\alpha + d}} \log ^{3t_1} n$ for sufficiently large $n$ if $\alpha > d/2$. Therefore, $(K_2 t)^{1/2} \leq n\epsilon_n$.
\end{proof}

{\em We next show that (BS) holds.} Fix $f \in \m P_n \cap U_n$. Since $f \in \m P_n$, there exists $h \in \mb H^{a_n}$ with $\norm{h}_{\mb H^{a_n}} \leq M_n$ such that $\|f - h\|_{\infty} \leq \epsilon_n$. By the triangle inequality, $\|\psi_{\sigma_n} \star f - f\|_{1} \leq \|\psi_{\sigma_n} \star f - \psi_{\sigma_n} \star h\|_{1} + \|\psi_{\sigma_n} \star h - h\|_{1} + \|h - f\|_{1}$. Using $\|\psi_{\sigma_n} \star g\|_1 \leq \|g\|_1$ for any $L_1$ function $g$, we can further bound $\|\psi_{\sigma_n} \star f - f\|_{1}$ from above by $2\epsilon_n + \|\psi_{\sigma_n} \star h - h\|_{\infty}$. It thus remains to show that $\|\psi_{\sigma_n} \star h - h\|_{\infty} \leq \epsilon_n$.

There exists $\xi \in L_2(\mu_{a_n})$ such that $h(t) = \int \exp( i \dotp{\lambda}{t} ) \xi(\lambda) \omega_{a_n}(\lambda) d\lambda$. Clearly, $\hat{h}(\lambda) = \xi(-\lambda) \omega_a(\lambda)$. Since the Fourier transform of $(\psi_{\sigma_n} \star h)$ is $(2 \pi)^d \hat{\psi}_{\sigma_n}  \hat{h}$ and $\hat{\psi}_{\sigma_n}(\lambda) = \hat{\psi}(\sigma_n \lambda)$, we have $\psi_{\sigma_n} \star h(t) = (2 \pi)^d \int \exp(-i \dotp{\lambda}{t}) \hat{\psi}(\sigma_n \lambda) \hat{h}(\lambda) d\lambda$. We can choose $\psi$ in a manner such that $\hat{\psi}$ is compactly supported, equals $(2 \pi)^{-d}$ on $[-1, 1]^d$
and is bounded above by this constant everywhere; see proof of Lemma 4.3 in \cite{van2009adaptive}. Putting together,
\begin{align*}
& | \psi_{\sigma_n} \star h(t) - h(t)|^2 \leq \bigg\{\int_{\|\lambda\| > \sigma_n^{-1}} |\hat{h}(\lambda)|\bigg\}^2 \leq \bigg\{\int \xi(\lambda)^2 \omega_{a_n}(\lambda) d\lambda\bigg\} \, \int_{\|\lambda\| > \sigma_n^{-1}} \omega_{a_n}(\lambda) d\lambda \\
& \leq C \norm{h}^2_{\mb H^{a_n}} \exp\{-\sigma_n^{-2}/(4a_n^2)\} \leq C M_n^2 \exp\{-\sigma_n^{-2}/(4a_n^2)\}  = C a_n^d \exp\{-(\log^{2t_2} n/4)\},
\end{align*}
where $C$ is an absolute constant. The proof follows by noting that $C a_n^d \exp\{-\log^{2t_2} n/4\} \leq
\epsilon_n^2$ whenever $t_2 > 1/2$ (holds for $t_2 =  1/(2- \kappa)$, for $0 < \kappa < 1$).
\subsection{Proof of Corollary \ref{cor:main}}
\noindent{\bf Case $\beta < \alpha$:}  Setting $\sigma_n = n^{-1/(2\beta + d)} \log^{-t_2}n$
for some constant $t_2 \geq 1/(2- \kappa)$, for $0 < \kappa < 2$, $M_n = a_n^{d/2}$, $\tilde{f}_n$ same as in \eqref{eq:fnhat}, $\mathcal{P}_n = M_n \mathbb{H}_1^{a_n} + \epsilon_n \mathbb{B}_1$ with $\epsilon_n = n^{-\beta/ (2\beta + d)} \log^{3t_1/2}n$, $t_1 \geq t_2d/2$, and $\delta_n =\epsilon_n = \epsilon_n'$,   one can verify
(PCS), (BT), (DT), (BS), (DS) exactly as in the proof of Theorem \ref{thm:main}.  (PCN) follows from Lemma 4.3 of \cite{van2009adaptive}.  \\
\noindent{\bf Case $\beta > \alpha$:}  Same as before with $\epsilon_n = n^{-\alpha/ (2\beta + d)} \log^{3t_1/2}n$ for $t_1 \geq t_2d/2$.

\section{Discussion}\label{sec:disc}
The article extends upon previous results on random design regression using Gaussian process priors.
A limitation of the current exposition is the requirement of the knowledge of the smoothness parameter to construct the rescaling sequence.  A natural question is whether one can find a suitable prior on the bandwidth parameter which adapts to the unknown smoothness level as in the fixed design case in \cite{van2009adaptive}.  We propose to resolve this issue as a part of future research. Also, our current proof technique would  lead to a sub-optimal rate of posterior convergence for $L_p$ norms with $p \neq 1$. We believe this is due to the use of Talagrand's inequality to construct the test function.  A key requirement to obtain optimal convergence rate is that the {\em variance term} $\sigma_{\mathcal{F}}^2$ in the application of Talagrand's inequality should be at most $O(n)$. This assertion is true only when $p=1$.   Obtaining  convergence rates for integrated $L_p$ norms with $p \neq 1$ is a topic of future research.

\section*{Acknowledgement}
Dr. Pati and Dr. Bhattacharya acknowledge support for this project
from the Office of  Naval Research (ONR BAA 14-0001). Dr. Cheng
acknowledges support from the National Science Foundation (NSF CAREER, DMS -- 1151692, DMS -- 1418042), Simons Foundation (305266) and warm hosting at SAMSI.

\appendix
\section*{Appendix}
\subsection*{Proof of Theorem \ref{thm:main_cond}}

Let $\|f\|_{2,n}$ denote the empirical $L_2$ norm of $f$, so that $\|f\|_{2,n}^2 = n^{-1} \sum_{i=1}^n f^2(X_i)$. Also, define
$$
L_n(f, f_0) = \frac{q^{(n)}(Y_{1:n}, X_{1:n} \mid f)}{q^{(n)}(Y_{1:n}, X_{1:n} \mid f_0)}.
$$
\begin{lemma}\label{lem:An}
Let $A_n$ denote the following event in the sigma-field generated by $(Y_{1:n}, X_{1:n})$:
\begin{align}\label{eq:An}
A_n = \bigg\{(Y_{1:n}, X_{1:n}) : \int L_n(f, f_0) \Pi(df) \geq e^{-n \delta_n^2} \Pi(\|f - f_0\|_{\infty} \leq \delta_n) \bigg\}.
\end{align}
Then, $\bbP_{Y, X}^{f_0}(A_n) \geq 1 - e^{-C n \delta_n^2}$.
\end{lemma}
\begin{proof}
Clearly, $\bbP_{Y,X}^{f_0}(A_n) = \bbE_X^{f_0} [\bbP_{Y \mid X}^{f_0}(A_n)]$. By Lemma 14 of \cite{van2011information}, $\bbP_{Y \mid X}^{f_0}\{ \int L_n(f, f_0) \Pi(df) \geq e^{-n \delta_n^2} \Pi(\|f - f_0\|_{2,n} \leq \delta_n)\} \geq 1 - e^{-n \delta_n^2/8}$. The conclusion follows by noting that $\Pi(\|f - f_0\|_{\infty} < \delta_n) \leq \Pi(\|f - f_0\|_{2,n} < \delta_n)$.
\end{proof}

\begin{lemma}\label{lem:test}
There exists a test function $\Phi_n$ for $H_0: f = f_0$ vs $H_1: f \in U_n \cap \m P_n$ such that
\begin{align}
& \bbE_{Y, X}^{f_0} \Phi_n \leq e^{-C n \delta_n^2},  \\
& \sup_{f \in U_n \cap \m P_n} \bbE_{Y, X}^{f} (1 - \Phi_n) \leq e^{-C n \delta_n^2}.
\end{align}
for some absolute constant $C$.
\end{lemma}
\begin{proof}
Let $\Phi_n = 1(\| \tilde{f}_n - f_0\|_1 > M \epsilon_n/2)$. The error bounds follow from (BT), (DT) and (BS), (DS). 
\end{proof}

Using a standard line of argument for establishing convergence rates in Bayesian non-parametric models \citep{ghosal2000convergence}, we have $\bbE_{Y, X}^{f_0} \Pi(U_n \mid Y_{1:n}, X_{1:n}) \leq \sum_{i=1}^4 b_{in}$, where $b_{1n} = \bbE_{Y, X}^{f_0} \Phi_n$, $b_{2n} = e^{n \delta_n^2} \sup_{f \in U_n \cap \m P_n} \bbE_{Y, X}^{f} (1 - \Phi_n) /\Pi(\|f - f_0\|_{\infty} < \delta_n)$, $b_{3n} =e^{n \delta_n^2} \Pi(\m P_n^c)/\Pi(\|f - f_0\|_{\infty} < \delta_n)$ and $b_{4n} = \bbP_{Y, X}^{f_0}(A_n^c)$.  The Theorem then follows from Lemmas \ref{lem:An}, \ref{lem:test} and Conditions (PCS) and (PCN).

\bibliographystyle{biometrika}
\bibliography{Xbib9}
\end{document}